\documentclass[final,11pt,a4paper]{amsart}
\usepackage{amsmath,amssymb,amsthm,amsfonts}
\usepackage[usenames,dvipsnames]{xcolor}
\usepackage[all]{xy}
\SelectTips{cm}{}
\usepackage{ifthen}
\usepackage{graphicx}
\usepackage{psfrag}
\usepackage{enumitem}
\usepackage{booktabs}
\usepackage{hyperref}

\hypersetup{
    colorlinks,
    linkcolor={red!50!black},
    citecolor={blue!50!black},
    urlcolor={blue!80!black}
}

\newboolean{ebookreader} 
\setboolean{ebookreader}{false}
\ifthenelse{\boolean{ebookreader}}{
\usepackage[paperwidth=160mm,paperheight=120mm,left=5mm,right=5mm,top=5mm,bottom =5mm]{geometry}
\setlength{\pdfpagewidth}{\paperwidth}
\setlength{\pdfpageheight}{\paperheight}
\pagestyle{empty}
\hfuzz=5mm
\emergencystretch=15mm}{
}

\newcommand{\tX}{{\tilde X}}

\newcommand{\Sph}[1]{($S^{#1}$)}        
\newcommand{\CY}[1]{($\text{CY}_{\!#1}$)} 
\newcommand{\SOD}{{\rm (\ddag)}}
\newcommand{\SOT}{{\rm (\dag)}}

\newcommand{\fixedwidthtabular}{ \noindent \begin{tabular}{@{} p{0.08\textwidth} @{} p{0.92\textwidth} @{} } }

\newcommand{\assumption}[2]{ \medskip\noindent
                             \begin{tabular}{@{} p{0.077\textwidth} @{} p{0.92\textwidth} @{} }
                               #1 & #2 
                             \end{tabular}\medskip }

\newcommand{\apppart}[1]{\bigskip\noindent\emph{#1:}}

\newcommand{\bib}[6]{{\bibitem{#2} #3, {\emph{#4},} #5#6.}}
\newcommand{\arXiv}[1]{{\href{http://arxiv.org/abs/#1}{\texttt{arXiv:#1}}}}

\newcommand{\TTT}{\mathsf{T}\!}  
\newcommand{\TTl}{\TTT^{\:l}}
\newcommand{\TTr}{\TTT^{\:r}}
\newcommand{\SSS}{\mathsf{S}}    

\newcommand{\Db}{\mathcal{D}^b}  
\renewcommand{\AA}{\mathcal{A}}
\newcommand{\BB}{\mathcal{B}}
\newcommand{\CC}{\mathcal{C}}
\newcommand{\DD}{\mathcal{D}}
\newcommand{\EE}{\mathcal{E}}

\newcommand{\HH}{\mathcal{H}}

\newcommand{\MM}{\mathcal{M}}
\newcommand{\NN}{\mathcal{N}}
\newcommand{\OO}{\mathcal{O}}

\newcommand{\QQ}{\mathcal{Q}}

\newcommand{\UU}{\mathcal{U}}

\newcommand{\IP}{\mathbb{P}}

\newcommand{\kk}{\mathbf{k}}

\renewcommand{\mod}{\mathsf{mod}}

\newcommand{\proj}{\mathsf{proj}}
\DeclareMathOperator{\add}{\mathsf{add}}
\newcommand{\hh}{\text{-}}  

\newcommand{\sod}[1]{{\langle #1 \rangle}}

\newcommand{\bigsod}[1]{{\big\langle #1 \big\rangle}} 

\newcommand{\dual}{^\vee}
\newcommand{\orth}{^\perp}
\newcommand{\op}{^\text{op}}
\newcommand{\lorth}{{}^\perp}
\newcommand{\inv}{^{-1}}
\newcommand{\blank}{\:\cdot\:} 

\newcommand{\coloneqq}{\mathrel{\mathop:}=}
\newcommand{\eqqcolon}{=\mathrel{\mathop:}}  

\DeclareMathOperator{\Hom}{Hom}
\DeclareMathOperator{\End}{End}

\DeclareMathOperator{\id}{id}

\DeclareMathOperator{\Pic}{Pic}

\DeclareMathOperator{\Cone}{Cone}

\DeclareMathOperator{\supp}{supp}

\DeclareMathOperator{\chara}{char}

\newcommand{\xyinjar}{\ar@{^{(}->}}
\newcommand{\xysurjar}{\ar@{->>}}
\newcommand{\arrd}{ \ar@{-}[r] \ar@{=}[d] }

\newcommand{\isom}{ \text{{\hspace{0.48em}\raisebox{0.8ex}{${\scriptscriptstyle\sim}$}}}
                    \hspace{-0.65em}{\rightarrow}\hspace{0.3em}} 
\newcommand{\embed}{\hookrightarrow}
\newcommand{\xxrightarrow}[1]{\xrightarrow{\raisebox{-0.1ex}{\ensuremath{{\scriptscriptstyle #1}}}}} 
\makeatletter
\providecommand*{\xhookrightfill@}{%
  \arrowfill@{\lhook\joinrel\relbar}\relbar\rightarrow
}
\providecommand*{\xembed}[2][]{%
  \ext@arrow 3095\xhookrightfill@{#1}{#2}%
}
\makeatother
\newcommand{\xxembed}[1]{\xembed{\raisebox{-0.1ex}{\ensuremath{{\scriptscriptstyle #1}}}}}

\newcommand{\onto}{\twoheadrightarrow}

\newtheorem{theorem}{Theorem}[section]

\newtheorem*{theorem*}{Theorem}
\newtheorem{proposition}[theorem]{Proposition}
\newtheorem*{proposition*}{Proposition}
\newtheorem{lemma}[theorem]{Lemma}
\newtheorem*{lemma*}{Lemma}

\theoremstyle{definition}
\newtheorem{definition}[theorem]{Definition}
\newtheorem{remark}[theorem]{Remark}
\newtheorem{corollary}[theorem]{Corollary}
\newtheorem{example}[theorem]{Example}
\newtheorem*{remark*}{Remark}
\newtheorem*{example*}{Example}
\newtheorem*{examples*}{Examples}
\newtheorem*{convention}{Conventions}

\newcommand{\sqmat}[4]{{\big(\genfrac{.}{.}{0pt}{1}{#1}{#3} \,
                        \genfrac{.}{.}{0pt}{1}{#2}{#4}\big) }}

\begin{document}

\title[Spherical subcategories in algebraic geometry]{Spherical subcategories\\in algebraic geometry}
\author{Andreas Hochenegger}
\author{Martin Kalck}
\author{David Ploog}

\begin{abstract}
We study objects in triangulated
categories which have a two-dimensional graded endomorphism algebra. Given such an object, we
show that there is a unique maximal triangulated subcategory, in which the object is 
spherical. This general result is then applied to examples from algebraic geometry. 
\end{abstract}

\maketitle

\thispagestyle{empty}

{
\small
\vspace{-4ex}             
\setcounter{tocdepth}{1}  
\tableofcontents          
\vspace{-4ex}
}

\section{Introduction} 

\noindent
Spherical objects were introduced by Seidel and Thomas \cite{Seidel-Thomas} to construct auto\-equivalences of triangulated categories. By definition, the Serre functor shifts such an object (Calabi-Yau property) and its graded endomorphism algebra is two-dimen\-sional.
Under Kontsevich's Homological Mirror Symmetry conjecture \cite{Kontsevich}, these autoequivalences are mirror dual to generalised Dehn twists about Lagrangian spheres. 
Some typical examples are structure sheaves of $(-2)$-curves on surfaces, and line bundles on Calabi-Yau varieties.

However, the Calabi-Yau property is in general not preserved under important fully faithful functors like pull-backs along blow-ups. This forces us to study \emph{spherelike objects} (i.e.\ objects with two-dimensional graded endomorphism algebras) in general. 

Moreover, the dimension of the graded endomorphism algebra may be viewed as a measure for complexity. Exceptional objects (i.e.\ objects with one-dimensional graded endomorphism algebras) have been extensively studied for several decades, see e.g.\ the seminal work of the Rudakov seminar \cite{Rudakov}. Our work may be regarded as a first step towards a general theory for spherelike objects which come next in terms of complexity.

The following general theorem is our main result. It shows that given a spherelike object in a triangulated category, there exists a unique maximal triangulated subcategory in which the object becomes spherical.

\begin{theorem*}[Theorems~\ref{thm:spherelike-CY} \& \ref{thm:maximality}]
Let $\DD$ be a Hom-finite $\kk$-linear triangulated category, and let $F\in\DD$ be a $d$-spherelike object possessing a Serre dual. Then there is a triangulated subcategory $\DD_F\subset\DD$ such that $F\in\DD_F$ is a $d$-spherical object, and if $\UU\subset\DD$ is a triangulated subcategory with $F\in\UU$ $d$-spherical, then $\UU\subset\DD_F$. We call $\DD_{F}$ the \emph{spherical subcategory}Êof $F$.
\end{theorem*}

We want to remark that the
spherical subcategory reflects important information of the ambient category, see e.g.\ the proposition and example below. 
Keller, Yang and Zhou \cite{Keller-Yang-Zhou} study the subcategory generated by a spherelike object and show that it does not depend on the ambient category --- in particular, it does not contain any geometric information. Their category is the minimal subcategory such that the object becomes spherical.

\smallskip

Our theory applies to many triangulated categories. For examples from the representation theory of finite-dimensional algebras, we refer to the forthcoming article \cite{HKP:RT}. In this article, applications to algebraic geometry are given in the final Section~\ref{sec:AG}. Here we mention Proposition~\ref{prop:structure-sheaf-spherelike} and Example~\ref{ex:structure-sheaf-ruled-surface}. 

\begin{proposition*}
Let $X$ be a Calabi-Yau variety and $\tX\to X$ be any succession of blow-ups in points. Then $\OO_\tX$ is a spherelike sheaf, and its spherical subcategory is $\Db(\tX)_{\OO_\tX}=\Db(X)$.
\end{proposition*}

It is well-known that the derived category of a blow-up has an explicit description as a semi-orthogonal decomposition; see \cite[\S11.2]{Huybrechts}. However, our result says that the spherical subcategory of the structure sheaf recovers the derived category of the original variety in one fell swoop. One thus might consider going from a category to a spherical subcategory as a kind of derived birational transformation.

\begin{example*}
Let $Y$ be a ruled surface over an elliptic curve. Then the structure sheaf is spherelike and the spherical subcategory $\Db(Y)_{\OO_Y}$ depends in an explicit manner on the choice of $Y=\IP(V)$ where $V$ is a rank two bundle on the elliptic curve; see Example~\ref{ex:structure-sheaf-ruled-surface} for details.
\end{example*}

\begin{convention}
All subcategories are assumed to be full.
The shift (or translation, or suspension) functor of triangulated categories is denoted by $[1]$. All triangles in triangulated categories are meant to be distinguished. Also, we will generally denote triangles abusively by $A\to B\to C$, hiding the degree increasing morphism $C\to A[1]$.

All functors between triangulated categories are meant to be exact. Derived functors are denoted with the same symbol as the (non-exact) functors between abelian categories. For example, if $f\colon X\to Y$ is a proper map of varieties, $f_*\colon\Db(X)\to\Db(Y)$ denotes the derived push-forward functor.
\end{convention}

\section{Preliminaries}

\noindent
In this section, we collect some terminology and basic facts. All of this is standard apart from the notion of a piecewise invertible functor.
Most of the material can be found in \cite{Huybrechts} unless otherwise stated.
Readers exclusively interested in spherical subcategories can fast forward to Section~\ref{sec:spherical_subcategories}.

Fix an algebraically closed field $\kk$. We write $(\blank)^* = \Hom_\kk(\blank,\kk)$ for dualisation over the ground field. By $\Hom^\bullet(A,B)$ we denote the complex $\bigoplus_i\Hom(A,B[i])[-i]$ of $\kk$-vector spaces with the zero differential. A $\kk$-linear triangulated category $\DD$ is called \emph{Hom-finite} if $\dim_\kk\Hom(A,B)<\infty$ for all objects $A,B$. It is called \emph{$\Hom^\bullet$-finite} if $\dim_\kk\Hom^\bullet(A,B)<\infty$ for all objects $A,B$ (\cite[Def.~3.12]{Orlov-proper} calls this \emph{proper}). The category $\DD$ is called \emph{idempotent complete} if any projector splits, i.e.\ for $p\in\Hom(A,A)$ with $p^2=p$ there is $A \cong A_1\oplus A_2$ such that $p$ is $A\to A_1\to A$. It is well-known that Hom-finite and idempotent complete categories are \emph{Krull-Schmidt}, i.e.\ every object has direct sum decomposition into indecomposable objects, unique up to order and isomorphism. For the convenience of the reader, we give the short argument below.

\begin{lemma}
Let $\DD$ be a $\kk$-linear Hom-finite, idempotent complete additive category. Then $\DD$ is Krull-Schmidt.
\end{lemma}

\begin{proof} Let $A\in\DD$. We denote by $\add(A)$ the smallest additive subcategory of $\DD$ containing all direct summands of $A$.
Put $\proj\End(A)$ for the category of finitely generated projective right modules over the ring $\End(A)$. The functor
 $\Hom(A,\blank)\colon\add(A)\to\proj\End(A)$
is an equivalence of additive categories since $\DD$ is idempotent complete.

Moreover, $\End(A)$ is in fact a finite-dimensional $\kk$-algebra. The Krull-Schmidt property holds for the category of finitely generated modules over an artinian ring (see \cite[\S{}II.2]{Auslander-Reiten-Smalo}), hence particularly holds for $\proj\End(A)$, and then also for $\add(A)$. Therefore $A$ decomposes essentially uniquely into a direct sum of indecomposable objects.
\end{proof}

\medskip

\subsubsection*{Serre duality}
Let $A\in\DD$ be an object. We say that $A$ \emph{has a Serre dual} if the cohomological functor $\Hom_\DD(A,\blank)^*$ is representable. If this is the case, the representing object is unique and will be denoted by $\SSS A$. By definition, we then get isomorphisms
 $\sigma_B\colon \Hom(B,\SSS A)\isom\Hom(A,B)^*$, functorial in $B\in\DD$.
Note that there is a canonical map $\Hom^\bullet(A,A)^*\to\kk$ and we claim that Serre duality implies that the pairing
\[ \Hom^\bullet(A,B) \otimes \Hom^\bullet(B,\SSS A) \to \Hom^\bullet(A,\SSS A) \isom \Hom^\bullet(A,A)^* \to \kk \]
is non-degenerate for all $B\in\DD$. This claim follows from 
 $\sigma_A(gf)(\id_A) = \sigma_B(g)(f)$ for all $f\colon A\to B$ and $g\colon B\to\SSS A$. These relations are formal consequences of the commutative diagrams
\[ \xymatrix@R=3ex{
  \Hom(A,\SSS A) \ar[r]^{\sigma_A}              & \Hom(A,A)^* \\
  \Hom(B,\SSS A) \ar[r]^{\sigma_B} \ar[u] & \Hom(A,B)^* \ar[u]
} \]
expressing functoriality of the $\sigma$ maps (the vertical maps are induced by $f$).

We say that the category $\DD$ \emph{has a right Serre functor} if every object $A$ has a Serre dual $\SSS A$ and write $\SSS\colon\DD\to\DD$ for the induced functor (by the functorial isomorphisms $\Hom(A,\blank)^*\isom \Hom(\blank,\SSS A)$ for any $A \in \DD$, we get natural transformations $\Hom(\blank,\SSS A)\to\Hom(\blank,\SSS B)$ for $f\colon A\to B$. By the Yoneda lemma this gives a unique $\SSS(f)\colon \SSS A \to \SSS B$ in turn; see \cite[Lem.~I.1.3(2)]{Reiten-vandenBergh}).
Here we follow \cite[\S{}I]{Reiten-vandenBergh}, where it is shown that a right Serre functor is always fully faithful but in general not essentially surjective.

We say that $\DD$ \emph{has a Serre functor} if it has a right Serre functor which is an equivalence. We will denote Serre functors by $\SSS$, unless $\SSS_\DD$ is needed for clarity. Serre functors are unique up to unique isomorphism, hence commute with equivalences. To be precise, if $\CC$ and $\DD$ are triangulated categories with Serre functors $\SSS_\CC$ and $\SSS_\DD$, and if $\varphi\colon\DD\isom\CC$ is an equivalence, then $\varphi\inv\SSS_\CC\varphi$ is a Serre functor for $\DD$, and hence has to be isomorphic to $\SSS_\DD$.

An object $A\in\DD$ is called a \emph{$d$-Calabi-Yau object}, for an integer $d$, if $A[d]$ is a Serre dual for $A$, so there are isomorphisms $\Hom^\bullet(A,B)\cong\Hom^\bullet(B,A[d])^*$, natural in $B\in\DD$.

The category $\DD$ is called a \emph{$d$-Calabi-Yau category}, if the shift $[d]$ is a Serre functor. We remark that a triangulated category might be $d$-Calabi-Yau for several numbers $d$ at once if a multiple of the shift functor is the identity. Also, it is not enough to demand that all objects are $d$-Calabi-Yau; see \cite[Ex.~9(1)]{Dugas} for a specific instance.

\subsubsection*{Spanning and generation}
Let $\Omega\subseteq\DD$ be a subset (or subclass) of objects. The left and right orthogonal subcategories of $\Omega$ are the full subcategories 
\begin{align*}
\Omega\orth   &= \{ A\in\DD \mid \Hom^\bullet(\omega,A) = 0 ~\forall \omega\in\Omega \} , \\
\lorth\Omega &= \{ A\in\DD \mid \Hom^\bullet(A,\omega) = 0 ~\forall \omega\in\Omega \} .
\end{align*}
Both of these are triangulated. We say that $\Omega$ \emph{spans} $\DD$ (or is a \emph{spanning class}) if $\Omega\orth=0$ and $\lorth\Omega=0$.

We denote by $\sod{\Omega}$ the smallest triangulated subcategory of $\DD$ closed under direct summands which contains $\Omega$; this is sometimes denoted by $\operatorname{thick}(\Omega)$.
We say that $\Omega$ \emph{classically} (or \emph{split}) \emph{generates} $\DD$ if $\sod{\Omega}=\DD$. We will omit the attribute ``classical'' in the subsequent text. A generating class is always spanning, but in general not vice versa.

\subsubsection*{Semi-orthogonal decompositions}
Essentially, the concepts here can be found in \cite{Bondal-Kapranov-SOD}, except the notion of weak semi-orthogonal decompositions which seems to be defined for the first time in \cite{Orlov-LG}. A triangulated subcategory $\NN \embed \DD$ is called \emph{left (or right) admissible} if the inclusion admits a left (or right) adjoint. To rephrase, $\NN$ is right admissible if for any $A \in \DD$ there is a unique functorial triangle $A_\NN \to A \to A_\perp$ with $A_\NN \in \NN$ and $A_\perp \in \NN\orth$.
In fact, $\DD\to\NN$, $A\mapsto A_\NN$ and $\DD\to\NN\orth$, $A\mapsto A_\perp$ are triangle functors by \cite[Prop.~1.3.3]{BBD}.
We call $\NN$ \emph{admissible} if $\NN$ is left and right admissible.

Actually, a pair $(\MM,\NN)$ of  triangulated subcategories of $\DD$, such that $\NN$ is right admissible and $\MM = \NN\orth$, is called a \emph{weak semi-orthogonal decomposition} of $\DD$.
Note that $\MM$ is automatically left admissible. If both $\MM$ and $\NN$ are additionally admissible then we call the pair a \emph{semi-orthogonal decomposition}. In both cases, we write $\DD = \sod{\MM,\NN}$. For readers more familiar with t-structures, we mention that a weak semi-orthogonal decomposition $\sod{\MM,\NN}$ is the same thing as a t-structure $(\NN,\MM)$ for which both subcategories $\MM$ and $\NN$ are triangulated.
Sometimes we will only write $\sod{\MM,\NN}$, where implicitly mean the category generated by the union of $\MM$ and $\NN$.

The definition can be extended inductively: a sequence $(\NN_1, \ldots, \NN_k)$  is a (weak) semi-orthogonal decomposition if $\sod{\sod{\NN_1, \ldots, N_{k-1}},\NN_k}$ is.

A special case are exceptional sequences. An object $E \in \DD$ is \mbox{\emph{exceptional}} if $\Hom^\bullet(E,E) = \kk$. A sequence of objects $(E_1, \ldots, E_k)$ is called \emph{exceptional} if all $E_i$ are exceptional and $\Hom^\bullet(E_j,E_i) = 0$ for $j>i$. Subcategories generated by exceptional sequence are admissible. In particular, if the exceptional sequence is \emph{full}, i.e.\ generates the whole category, then $\DD = \sod{E_1,\ldots,E_k}$ is a semi-orthogonal decomposition; by common abuse of notation we write $E_i$ to mean the triangulated category generated by the exceptional object $E_i$.

\subsubsection*{Adjoints}
Let $\varphi\colon\CC\to\DD$ be an exact functor between triangulated categories. If $\varphi$ has a right adjoint, it will be denoted by $\varphi^r\colon\DD\to\CC$. It is a simple fact that $\varphi^r$ is again exact; \cite[Prop.~1.41]{Huybrechts}. The same holds for a left adjoint $\varphi^l$.
%
%
The next lemma collects two well known and simple properties of adjoints:

\begin{lemma} \label{lem:adjoint_properties}
Let $\varphi\colon\CC\to\DD$ be an exact functor between triangulated categories with a right adjoint $\varphi^r$. Assume that $\DD$ has a Serre functor $\SSS_\DD$.

\fixedwidthtabular
(1) & If $\varphi$ is fully faithful, i.e.\ $\CC$ right admissible in $\DD$, then
      $\SSS_\CC = \varphi^r\SSS_D\varphi$ is a right Serre functor for $\CC$.
\\
(2) & If $\CC$ has a Serre functor, then there is a left adjoint 
      $\varphi^l=\SSS_\CC\inv \varphi^r \SSS_\DD$. \\
\end{tabular}
\end{lemma}

\begin{proof}
(2) is straightforward. For (1), compute for any objects $A,B \in \CC$
\begin{align*}
 \Hom_\CC(A,\SSS_\CC B) &= \Hom_\CC(A,\varphi^r\SSS_\DD\varphi B) 
                         = \Hom_\DD(\varphi A,\SSS_\DD\varphi B) \\
                        &= \Hom_\DD(\varphi B,\varphi A)^* 
                         = \Hom_\CC(B,A)^* . \qedhere 
\end{align*}
\end{proof}

\subsubsection*{Functor properties}
We list some properties a functor might enjoy and which equivalences always have. All notions are standard apart from the last one.

\begin{definition} \label{defn:functorproperties}
Let $\varphi\colon\DD\to\DD'$ be an exact functor between $\kk$-linear, triangulated categories. Then $\varphi$ is said to be

\renewcommand{\descriptionlabel}[1]{%
  \hspace\labelsep \upshape #1%
}
\begin{description}
\item[\emph{fully faithful}] if the maps $\Hom_\DD(D_1,D_2)\isom\Hom_{\DD'}(\varphi(D_1),\varphi(D_2))$ induced by $\varphi$ are isomorphisms for all $D_1,D_2\in\DD$.
\item[\emph{conservative}] if $f$ is a morphism in $\DD$ such that $\varphi(f)$ an isomorphism, then $f$ is an isomorphism itself.
\item[\emph{essentially surjective}] if for any object $D'\in\DD'$, there is an object $D\in\DD$ such that $\varphi(D)\cong D'$.
\item[an \emph{equivalence}] (or \emph{invertible}) if $\varphi$ is fully faithful and essentially surjective.
\item[\emph{piecewise invertible}] if there are weak semi-orthogonal decompositions $\DD = \sod{\DD_1,\ldots,\DD_n}$, $\DD' = \sod{\DD'_1,\ldots,\DD'_n}$ such that $\varphi(\DD_i)\subseteq\DD'_i$ and $\varphi|_{\DD_i}$ induces an equivalence $\DD_i\isom\DD'_i$.
\end{description}
\end{definition}

\begin{lemma} \label{lem:piecewise-invertible-is-conservative}
A piecewise invertible functor $\varphi\colon\DD\to\DD'$ is conservative.
\end{lemma}

\begin{proof}
We choose weak semi-orthogonal decompositions $\DD = \sod{\DD_1,\DD_2}$ and $\DD' = \sod{\DD'_1,\DD'_2}$ such that $\varphi$ induces equivalences of each component --- more than two components can be dealt with by induction. 

Given a morphism $f\colon A\to B$ in $\DD$ such that $\varphi(f)$ is an isomorphism, we consider the triangle $A\to B\to C$ where $C=\Cone(f)$. Through the triangle functors $\DD\to\DD_i$ associated with the semi-orthogonal decomposition, we obtain a commutative diagram
\[ \xymatrix@R=3ex{
A_2 \ar[d]^{f_2} \ar[r] & A \ar[d]^f \ar[r] & A_1 \ar[d]^{f_1} \\
B_2 \ar[d]       \ar[r] & B \ar[d]   \ar[r] & B_1 \ar[d]      \\
C_2              \ar[r] & C          \ar[r] & C_1
} \]
with $A_i,B_i\in\DD_i$ and where the rows and columns are exact. Applying $\varphi$ to the whole diagram, we find $\varphi(C)=0$ since $\varphi(f)$ is an isomorphism by assumption. Thus $\varphi(C_1)\cong\varphi(C_2)[1]$ lives in $\DD'_1\cap\DD'_2=0$, and we get $\varphi(C_1)=0$, $\varphi(C_2)=0$. As $\varphi$ induces equivalences $\DD_1\isom\DD'_1$ and $\DD_2\isom\DD'_2$, we deduce $C_1=0$ and $C_2=0$. Hence $f_1$ and $f_2$ are isomorphisms and then $f$ is an isomorphism, as well.
\end{proof}

Note that the composition of piecewise invertible functors is not necessarily piecewise invertible again, whereas the other four properties of the definition are closed under composition. Let us recall standard criteria for fully faithfulness and for equivalence:

\begin{proposition}[{\cite[Prop.~1.49]{Huybrechts}}] \label{prop:fully-faithful-criterion}
Assume that $\varphi$ has left and right adjoints and let $\Omega\subseteq\DD$ be a spanning class. Then $\varphi$ is fully faithful if and only if $\varphi|_\Omega$ is fully faithful.
\end{proposition}

\begin{lemma}[{\cite[Lem.~1.50]{Huybrechts}}] \label{lem:surjective-criterion}
Let $\varphi\colon \DD\to\DD'$ be a fully faithful functor with right adjoint $\varphi^r$. Then $\varphi$ is an equivalence if and only if $\varphi^r(C) = 0$ implies $C=0$ for all $C \in \DD'$.
\end{lemma}

\section{Twist functors} \label{sec:twist-functors}

\noindent
Let $\DD$ be a $\kk$-linear, Hom-finite triangulated category and $d$ an integer. For an object $F\in\DD$ we consider the following two properties:

\noindent 
\begin{tabular}{@{} p{0.1\textwidth} @{} p{0.9\textwidth} @{} }
\Sph{d} & $\Hom^\bullet(F,F) = \kk\oplus\kk[-d]$, i.e.\ the only non-trivial derived endomorphism besides the identity is a $d$-extension $F\to F[d]$, up to scalars. \\
\CY{d}  & $F$ is a $d$-Calabi-Yau object, i.e.\ $\Hom^\bullet(F,A)\cong\Hom^\bullet(A,F[d])^*$, functorially in $A\in\DD$.
\end{tabular}

$F$ is called \emph{$d$-spherelike} if it satisfies \Sph{d}. Is it called \emph{$d$-spherical} if it satisfies both \Sph{d} and \CY{d}. 
The number $d$ may be dropped when it is clear from the context or not relevant. We say that $F$ is \emph{properly spherelike} if it is spherelike but not spherical.

Non-positive numbers $d$ are allowed. Often, the case $d=0$  needs special attention. For example, given a $d$-spherical object $F$ with $d\neq0$ it is obvious that the Serre dual of an isomorphism $\Hom(F,F)\isom\kk$ is a non-trivial extension $F\to F[d]$ which can simplify arguments. Because most emphasis is on positive $d$, in the main text we will state general results (including $d=0$) but will defer the proofs for $d=0$ to the Appendix.
We want to note that for $d=0$, our definition is slightly broader than the one in \cite{Seidel-Thomas}. There they also ask for $\End^\bullet(F) \cong \kk[x]/x^2$, e.g.\ they exclude decomposable objects; see also the Appendix.

\subsection{Algebraic triangulated categories and functorial cones}

In order to define the twist functors in great generality, we have to make a rigidity assumption on our triangulated categories: they should be \emph{enhanced} in the sense of Bondal and Kapranov \cite{Bondal-Kapranov-enhanced} or \emph{algebraic} in the sense of Keller \cite{Keller-dg}; these notions are equivalent by \cite[Thm.~3.8]{Keller-dg}.

More precisely, an \emph{enhancement} of a $\kk$-linear triangulated category $\DD$ is a pretriangulated differential graded (dg) category $\AA$ together with an equivalence $H^0(\AA)\isom\DD$ of triangulated categories. We refer to 
\cite[\S4.5]{Keller-dg} or \cite[\S4.4]{Toen} for the notion of a pretriangulated dg-category (note that \cite{Toen} calls these triangulated dg-categories, provided the underlying derived categories are idempotent complete). 
In other words, for objects $X,Y$ of $\AA$, the homomorphism spaces $\Hom_\AA(X,Y)$ are dg $\kk$-modules, i.e.\ complexes of $\kk$-vector spaces. The associated homotopy category $H^0(\AA)$,  denoted $[A]$ in \cite{Toen}, has the same objects as $\AA$ and morphisms $H^0(\Hom_\AA(X,Y))$. It is a $\kk$-linear category, and triangulated since $\AA$ is pretriangulated.


This assumption implies that cones of morphisms are functorial in the following sense: given two dg-categories $\AA$ and $\BB$ with $\BB$ pretriangulated, then the category $\HH om(\AA,\BB)$ of of dg-functors $\AA\to\BB$ is itself a pretriangulated dg-category --- hence has cones (loc.\ cit.\ in \cite{Keller-dg,Toen}). 
The catchphrase about `functorial cones in $\DD\cong H^0(\AA)$' then means the following.
Every dg-functor $\tilde\varphi\colon \AA \to \AA$ gives rise to a triangle functor $\varphi\colon \DD \to \DD$ by taking $H^0$; see Proposition 10 and the following remark in \cite{Toen}. And every natural transformation of dg-functors $\tilde\varphi \to \tilde\varphi'$ gives rise to a natural transformation of triangle functors $\varphi \to \varphi'$.  
Starting with triangle functors $\varphi,\varphi' \colon \DD \to \DD$ and a natural transformation $\nu\colon \varphi \to \varphi'$ between them, which admit lifts $\tilde\varphi$, $\tilde\varphi'$ and $\tilde\nu$ for an enhancement $\AA$, we obtain a dg-functor $\Cone(\tilde \nu)$, since $\AA$ has functorial cones.
Then the triangle functor $H^0(\Cone(\tilde\nu))$ fits into the triangle:
\[
\varphi \to \varphi' \to H^0(\Cone(\tilde\nu))
\]
See \cite[\S5.1]{Toen} for details.

All triangulated categories occurring in this article and its representation-theoretic counterpart \cite{HKP:RT} are of this type: bounded derived categories of abelian categories with enough injectives have an enhancement by \cite[\S3]{Bondal-Kapranov-enhanced}. In fact, by \cite{Spaltenstein} and \cite{Lunts-Orlov} enhancements also exist for the derived category $\Db(X)$ of coherent sheaves on a quasi-projective scheme, even though there are no injectives and these enhancements are, moreover, unique.

\subsection{General twist functors and adjoints}

Assuming that $\DD$ is $\Hom^\bullet$-finite, we will associate to an object $F\in\DD$ a \emph{twist functor} $\TTT_F\colon\DD\to\DD$. We begin by considering the exact functor $\Hom^\bullet(F,\blank)\colon\DD\to\Db(\kk)$; it is well-defined because $\DD$ is $\Hom^\bullet$-finite.
Next, we get an induced exact functor $\Hom^\bullet(F,\blank)\otimes F\colon\DD\to\DD$, together with a natural transformation $\Hom^\bullet(F,\blank)\otimes F\to\id$ coming from the evaluation $\Hom^\bullet(F,A)\otimes F \to A$. 
We would like to define the twist functor by the following exact triangle
\[ \Hom^\bullet(F,\blank)\otimes F \to \id \to \TTT_F . \]
For the functoriality we assume that $\DD$ is algebraic and idempotent complete. Let $\AA$ and $\mathcal V$ be enhancements of $\DD$ and $\Db(\kk)$, respectively. Choose a dg lift of $\blank \otimes F \colon \Db(\kk) \to \DD$ to a dg functor $\mathcal V \to \AA$.
By \cite[\S2.2]{Anno-Logvinenko}, there are also canonical lifts for its adjoints, especially for $\Hom^\bullet(F,\blank)$, and canonical lifts of the corresponding adjunction units and counits. The latter gives a lift of the evaluation $\Hom^\bullet(F,\blank)\otimes F \to \id$ to $\AA$.
As discussed above, there is a dg-functor $\tilde \TTT_F$ completing the lifted evaluation to an exact triangle, so we can define $\TTT_F$ as $H^0 (\tilde\TTT_F)$.
This construction is well known; see \cite{Seidel-Thomas} and the much more general \cite{Anno-Logvinenko}. For a definition of twist functors using Fourier-Mukai kernels, see \cite[\S8]{Huybrechts}.

We mention in passing that the functor $\Hom^\bullet(F,\blank)\otimes F\colon\DD\to\DD$ --- and hence the twist functors --- can exist in greater generality; it suffices that $\Hom^\bullet(F,A)\otimes F$ exist for all $A\in\DD$. Our assumption that $\DD$ is $\Hom^\bullet$-finite ensures this.

\medskip\noindent
\emph{Throughout the rest of the article, whenever twist functors are mentioned we presume that $\DD$ is algebraic, idempotent complete and $\Hom^\bullet$-finite.} 
\medskip

We are going to describe the adjoints of $\TTT_F$. The left adjoint exists in full generality; the right adjoint needs a Serre dual $\SSS F$ of $F$. For any object $G\in\DD$, the endofunctor $\Hom^\bullet(F,\blank)\otimes G$ has adjoints
\[ \begin{array}{ *{3}{r@{\:}c@{\:}l c} }
   \Hom^\bullet(\blank,G)^* &\otimes& F        & \dashv &
   \Hom^\bullet(F,\blank)   &\otimes& G        & \dashv &
   \Hom^\bullet(G,\blank)   &\otimes& \SSS F
\end{array} \]
Note that a triangle of functors $\varphi\to\psi\to\eta$ leads to a triangle $\eta^l\to\psi^l\to\varphi^l$ of their left adjoints. 
To see this, apply first $\Hom^\bullet(A,\blank)$ and then adjunction to $\varphi(B)\to\psi(B)\to\eta(B)$ to yield a triangle $\Hom^\bullet(\varphi^l(A),B)\to\Hom^\bullet(\psi^l(A),B)\to\Hom^\bullet(\eta^l(A),B)$ in $\Db(\kk)$, functorial in both $A$ and $B$. 
This triangle is induced, thanks to the Yoneda lemma, by natural transformations $\psi^l \to \varphi^l$ and $g\colon\eta^l\to\psi^l$.
Next we apply $\Hom^\bullet(\blank,B)$ to the completed triangle
 $\eta^l(A) \smash{\xxrightarrow{g_A}} \psi^l(A) \to C(g_A)$, 
and deduce isomorphisms $C(g_A) \isom \varphi^l(A)$, as before by the Yoneda lemma. This shows that $\eta^l(A)\to \psi^l(A) \to \varphi^l(A)$ is again a triangle, whose functoriality is immediate.
The analogous statement for the right adjoints holds likewise. For the twist functor $\TTT_F$ under consideration, we get
\[
\TTl_F \to \id \to \Hom^\bullet(\blank,F)^*\otimes F \quad\text{and}\quad
\TTr_F \to \id \to \Hom^\bullet(F,\blank)\otimes \SSS F.
\]
We will prove in Lemma~\ref{lem:spanningprop} that for a spherical twist the left and right adjoints coincide and give the inverse. For a properly spherelike object $F$, the adjoints are necessarily distinct.

\subsection{Special cases of twist functors} \label{sub:special-cases}

The twist functors are most interesting when the derived endomorphism algebras are small:

\subsubsection*{Zero object}
Clearly, $\TTT_F = \id$ for $F=0$. From now on, assume $F$ non-zero.
\subsubsection*{Exceptional objects}
An exceptional object $F$, i.e.\ $\Hom^\bullet(F,F)=\kk$, is one with the smallest derived endomorphism ring. Each such object yields two semi-orthogonal decompositions $\sod{F\orth,F}$ and $\sod{F,\lorth F}$ of $\DD$. Furthermore, the twist functor $\TTT_F$ is a right adjoint of the inclusion $F\orth\embed\DD$; the shifted functor $\TTT_F[-1]$ is just the left mutation along $F$, as in \cite[\S7.2.2]{Rudakov}. An exceptional object is typically not studied for its own sake. Rather, one is looking for a full exceptional sequence --- a `basis' for the category --- or tries to strip off an exceptional object or sequence, by considering the orthogonal complement. See \cite{Rudakov} for many geometric examples of this approach.
\subsubsection*{Spherelike objects}
If $\Hom^\bullet(F,F)$ is two-dimensional, then by definition $F$ is spherelike. This is the next simplest case after exceptional objects, in terms of complexity. Spherical objects can be characterised as the simplest type of Calabi-Yau objects (leaving trivial examples aside, like $\kk$ in $\Db(\kk\hh\mod)$).

A spherical object $F$ is interesting on its own, since the associated twist functor $\TTT_F$ is an autoequivalence of the category \cite{Seidel-Thomas}. Collections of spherical objects provide interesting subgroups of autoequivalences; a topic related to (generalised) braid group actions and also first taken up in \cite{Seidel-Thomas} and independently in \cite{Rouquier-Zimmermann}.

\subsection{Spherelike twist functors}

In this article, we will deal exclusively with spherelike objects and show that their twist functors still have some interesting properties even though they are fully faithful only if the object is already spherical. Remarkably, an abstract spherelike object becomes spherical in a naturally defined subcategory. We start by giving a number of basic properties of twist functors, somewhat more careful than in \cite{Huybrechts} or \cite{Seidel-Thomas}, as we are not only interested in autoequivalences.

In the following lemma, we write $\Hom^\bullet_\circ(F,F)$ for the complex of \emph{traceless derived endomorphisms} of an object $F$; it is defined as the cone of the natural map $\kk\cdot\id_F\to\Hom^\bullet(F,F)$.

\begin{lemma} \label{lem:spanningprop}
Let $F\neq0$ be an object of $\DD$.

\fixedwidthtabular
(1) & $\TTT_F|_{F\orth} = \id$ and $\TTT_F|_\sod{F} = [1]\otimes\Hom^\bullet_\circ(F,F)$. \\
(2) & $\Hom^\bullet(F,F)\isom\Hom^\bullet(\TTT_F(F),\TTT_F(F))$ if and only if $F$ is spherelike. \\
(3) & If $\TTT_F$ is fully faithful, then $\TTT_F$ is an equivalence.\\
(4) & If $F$ is spherical, then $\TTT_F$ is an equivalence.\\
(5) & If $\TTT_F$ is an equivalence and $\DD$ has a Serre functor and is Krull-Schmidt,
      then $F$ is spherical.
\end{tabular}
\end{lemma}

\begin{proof}

(1) This follows at once from the defining triangle for twist functors.

(2) The condition means that $\TTT_F$ is fully faithful on the singleton $\{F\}$. 
The short proof for the ``if''-part can be found in \cite[Thm.~1.27]{Ploog-PhD}, which we replicate for the convenience of the reader.
For any $f\in\Hom^d(F,F)$,
\[
\xymatrix@R=5ex@C=5em{
F[-d]         \ar[r]^-{(f[-d],-\id)}    \ar@{..>}[d]_{\TTT_F(f)[-1]} & 
F\oplus F[-d] \ar[r]^-{\id\oplus f[-d]} \ar[d]^{\sqmat{f}{0}{0}{f[-d]}}       &
F \ar[d]^f    \ar[r]^0                                             &
F[1-d] \ar@{..>}[d]^{\TTT_F(f)}\\
F \ar[r]_-{(f,-\id)} & F[d]\oplus F \ar[r]_-{\id\oplus f} & F[d] \ar[r]_0 & F[1]
}
\]
is a commutative diagram of exact triangles.
Thus $\TTT_F(f)=f[1-d]$ and hence $\Hom^i(F,F) \isom \Hom^i(\TTT_F(F),\TTT_F(F))$ for $i=d$. For $i \neq d$, both sides are zero.
 
For the converse, note that $\Hom^\bullet_\circ(F,F)$ has to be one-dimensional, so $F$ is spherelike.

(3) We show now that $\TTl_F(D) \cong 0$ implies $D\cong 0$. Then we can apply the result analogous to Lemma~\ref{lem:surjective-criterion} for left adjoints to deduce that $\TTT_F$ is an equivalence:
Assuming $\TTl_F(D) \cong 0$, the triangle defining $\TTl_F$ boils down to $D \cong \Hom^\bullet(D,F)^* \otimes F$. Applying $\TTl_F$ to this isomorphism, we find $0\cong \Hom^\bullet(D,F)^*\otimes\TTl_F(F)$. Since $\TTl_F(F) = F[d-1]$, we get $\Hom^\bullet(D,F)^*\cong0$ and hence $D \cong \Hom^\bullet(D,F)^* \otimes F \cong 0$.

(4) Assume that $F$ is spherical. 
We start by showing that $\Omega\coloneqq \{F\} \cup F\orth$ is a spanning class for $\DD$: 
The property $\Omega\orth=0$ follows immediately from the construction of $\Omega$. The other vanishing $\lorth\Omega=0$ uses $\SSS F = F[d]$, i.e.\ $F$ is a Calabi-Yau object.

We claim that the maps $\Hom^\bullet(A,A')\to\Hom^\bullet(\TTT_F(A),\TTT_F(A'))$ induced by $\TTT_F$ are isomorphisms for all $A,A'\in\Omega$. 
This is true for $A=A'=F$ by (2). It holds for $A,A'\in F\orth$ as $\TTT_F|_{F\orth}$ is the identity. Both Hom spaces vanish if $A=F$ and $A'\in F\orth$. Finally, we also have vanishing in the remaining case $A\in F\orth$ and $A'=F$ --- here we invoke $\SSS F = F[d]$ again: $\Hom^\bullet(A,F)=\Hom^\bullet(F,A[-d])^*=0$. Thus $\TTT_F$ is fully faithful on the spanning class $\Omega$ and then fully faithful altogether by Proposition~\ref{prop:fully-faithful-criterion}. Using (3) we are done.

(5) Now assume that $F$ is an object such that $\TTT_F$ is an equivalence. We start with the observation that for any fully faithful functor $\varphi$, there is a natural transformation $\varphi\to\SSS\varphi\SSS\inv$. If $\varphi$ is an equivalence, the transformation is a functor isomorphism.

We look at the triangle defining the twist and at its Serre conjugate:
\[ \xymatrix@R=3ex{
   \Hom^\bullet(F,\blank) \otimes F   \ar@{..>}[d] \ar[r] & \id   \ar[d] \ar[r] & \TTT_F    \ar[d] \\
   \Hom^\bullet(F,\SSS\inv(\blank)) \otimes \SSS F \ar[r] & \SSS\SSS\inv \ar[r] & \SSS\TTT_F\SSS\inv
} \]
The two right-hand vertical maps define the one on the left. As $\id$ and $\TTT_F$ are equivalences, these functor maps are actually isomorphisms, hence the left-hand map is as well. Plugging in $\SSS F$, we get 
\[ \Hom^\bullet(F,\SSS F)\otimes F\isom\Hom^\bullet(F,F)\otimes\SSS F . \]
Since we already know from (2) that $F$ is $d$-spherelike for some $d$, we thus get
$F\oplus F[d] \cong \SSS F\oplus \SSS F[-d]$. 
 
For $d\neq0$, we note that $F$ is indecomposable from $\Hom(F,F)=\kk$. Hence $F \cong \SSS F[-d]$, as $\DD$ is Krull-Schmidt. For $d=0$, we refer to the Appendix.
\end{proof}

\begin{remark}
If we just assume that $\DD$ is algebraic and idempotent complete and moreover $\Hom^\bullet(A,F)$ and $\Hom^\bullet(F,A)$ are finite-dimensional for all $A \in \DD$, then the statements (4) and (5) of the preceeding lemma still hold by \cite[Thm.~5.1]{Anno-Logvinenko}.
\end{remark}

So far, we have solely considered $\TTT_F$ as an endofunctor of $\DD$. In the following, we will also take into account subcategories of $\DD$ inheriting the twist functor. The subsequent lemma shows a dichotomy for such subcategories. Note that $F$ will be spherelike in any subcategory containing it (recall that our subcategories are always full). However, $F$ might become spherical in a suitable subcategory and in fact, in the next section we will look for the maximal subcategory containing $F$ on which $\TTT_F$ becomes an equivalence.

\begin{lemma}
\label{lem:subcat-dicho}
Let $\UU\subseteq\DD$ be a  triangulated subcategory which is closed under taking direct summands. The twist functor $\TTT_F$ induces an endofunctor of $\UU$, i.e.\ $\TTT_F(\UU)\subseteq\UU$, if and only if either $F\in\UU$ or $\UU\subseteq F\orth$. More precisely:

\fixedwidthtabular
(1) &  If $F\in\UU$, then 
$\TTT_F\colon\UU\to\UU$ exists and coincides with the restriction of $\TTT_F\colon\DD\to\DD$ to $\UU$.\\
(2) & If $\UU \subseteq F\orth$, then the induced endofunctor $\TTT_F|_\UU$ is $\id_\UU$.\\
\end{tabular}
\end{lemma}

\begin{proof}
Only one implication of the equivalence is not obvious. So assume $\TTT_F(\UU)\subseteq\UU$ and pick $U\in\UU$. As the last two objects of the exact triangle $\Hom^\bullet(F,U)\otimes F\to U\to\TTT_F(U)$ are in $\UU$, we find $\Hom^\bullet(F,U)\otimes F\in\UU$ as well. Since $\UU$ is closed under summands, this boils down to $F\in\UU$ or $\Hom^\bullet(F,U)=0$. We are done, since the existence of a single $U\in\UU$ with $U\notin F\orth$ forces $F\in\UU$ by the same reasoning.

For (1), let $F\in\UU$ and $\UU$ be a full subcategory.
Let $\AA$ be an enhancement of $\DD$, especially it is a pre-triangulated dg category. Since $\UU$ is a full triangulated subcategory, the preimage of $\UU$ under the projection $\AA \to H^0(\AA)=\DD$ is an enhancement of $\UU$; see \cite[\S1.5.5]{Beilinson-Vologodsky}. So $\UU$ is algebraic. Since $\UU$ contains $F$, the twist functor exists.
Using these enhancements, the triangles defining the twist functor for $\UU$ and the one for $\DD$ can be lifted in a compatible manner to the respective enhancements.

(2) The claim for $\UU \subseteq F\orth$ follows from Lemma~\ref{lem:spanningprop}(1).
\end{proof}

\section{Spherical subcategories} \label{sec:spherical_subcategories}

\noindent
In this section, we are going to associate to a spherelike object $F$ a canonical subcategory $\DD_F$ where it becomes spherical. Therefore we call $\DD_F$ the \emph{spherical subcategory} of $F$. Before that we need a `measure' for the asphericity of $F$. Recall that $F$ is $d$-spherelike if $\Hom^\bullet(F,F)\cong\kk\oplus\kk[-d]$.

\subsection{The asphericity triangle} \label{sub:asphericity}

To a $d$-spherelike object $F\in\DD$ with Serre dual, we will associate a canonical triangle, the \emph{asphericity triangle}
\[  F \xrightarrow{w} \omega(F) \to Q_F \]
whose last term $Q_F$ is called the \emph{asphericity} $Q_F$ of $F$. 
It measures how far $F$ is from being spherical. If the object $F$ is clear from the context, we will often write  $Q$ in place of $Q_F$.

We begin by putting $\omega(F) \coloneqq  \SSS F[-d]$. The notation is borrowed from algebraic geometry; see Section~\ref{sec:AG}. Then we find
\[ \Hom^\bullet(F,\omega(F)) = \Hom^\bullet(F,\SSS F[-d]) = 
   \Hom^\bullet(F,F)^*[-d] = \kk \oplus \kk[-d] .\]
So for $d\neq0$, there is a non-zero map $w\colon F \to \omega(F)$, unique up to scalars. For $d=0$, the construction is slightly more elaborate and we give it in the Appendix.

\begin{lemma} \label{lem:Hom(F,Q)=0}
For a $d$-spherelike $F$ with Serre dual we have $\Hom^\bullet(F,Q) = 0$, i.e.\ $F\in\lorth Q$.
\end{lemma}

\begin{proof}
We apply $\Hom^\bullet(F,\blank)$ to the triangle $F \xxrightarrow{w} \omega(F) \to Q$ and obtain
\[ \Hom^\bullet(F,F) \xrightarrow{w_*} \Hom^\bullet(F,\omega(F)) \to \Hom^\bullet(F,Q) \]
where $\Hom^\bullet(F,F)$ and $\Hom^\bullet(F,\omega(F))$ are isomorphic to $\kk\oplus\kk[-d]$. Obviously $w_*(\id)=w$. For $d\neq0$, denote non-zero $d$-extensions $\varepsilon\colon F[-d]\to F$ and $\eta\colon F[-d]\to\omega(F)$. We look at the pairing
\[ \Hom(F[-d],F) \otimes \Hom(F,\omega(F)) \to \Hom(F[-d],\omega(F)), \quad 
     \varepsilon \otimes w             \mapsto w\circ\varepsilon \]
which is non-degenerate by Serre duality. Now all three Hom-spaces are one-dimensional, so that $w\circ\varepsilon$ is a non-zero multiple of $\eta$. Hence $w_*$ is an isomorphism and thus $\Hom^\bullet(F,Q) = 0$, as desired.

For the case $d=0$ we refer to the Appendix.
\end{proof}

\begin{remark} \label{rmk:Q-not-left-orthogonal}
For $F$ not spherical, i.e.\ $Q\neq0$, the marked morphism in the shifted triangle $\omega(F)\to Q\xrightarrow{*} F[1]$ is non-zero.
Otherwise $\omega(F) \cong F \oplus Q$, hence $\End^\bullet(\omega(F)) \cong \Hom^\bullet(F[-d],\omega(F))^* \cong \Hom^\bullet(F[-d],F\oplus Q)^* \cong \End^\bullet(F)^*[d]$. 
This is absurd since $\End^\bullet(\omega(F)) \supseteq \End^\bullet(F) \oplus \End^\bullet(Q)$.

So the marked map is non-zero, thus $Q$ is not left orthogonal to $F$, i.e.\ $Q\in F\orth\backslash\lorth F$. In particular for $d\neq1$, the asphericity always spoils fully faithfulness of $\TTT_F$, in view of
 $\Hom^\bullet(Q,F) \to \Hom^\bullet(\TTT_F(Q),\TTT_F(F)) = \Hom^\bullet(Q,F)[1-d]$
with $\Hom^1(Q,F)\neq0$ (note $Q \in F\orth$ implies $\TTT_F(Q)=Q$).
\end{remark}

\begin{remark}[{cf.~\cite[\S1a]{Seidel-Thomas}}]
\label{rem:Kgroup}
Give a $2d$-spherelike object $F\in\DD$ with its twist functor $\TTT_F\colon \DD\to\DD$, the endomorphism $t_F\coloneqq\TTT_F^{\:K} \colon K(\DD)\to K(\DD)$ of the $K$-group of $\DD$ is an involution with the propery $t_F([F]) = -[F]$.
This follows immediately from $t_F(x) = x -\chi(F,x)[F]$ for all $x\in K(\DD)$ and unravelling $t_F^2(x)$, using $\chi(F,F)=2$ as $F$ is $2d$-spherelike. 
Here, $\chi([F],[G]) = \chi(F,G) \coloneqq \sum_i (-1)^i \Hom(F,G[i])$ is the Euler pairing.

Thus spherelike objects might stand behind $K$-group involutions of geometric interest, as they allow lifts to endofunctors, even though not auto\-equivalences in general. Note that if $\DD$ is a $2d$-Calabi-Yau category, then the involution $t_F$ is actually the reflection along the root $[F]\in K(\DD)$.

Furthermore, these involutions satisfy the braid relations. More precisely, let $E,F\in\DD$ be spherelike with $\chi(E,E)=\chi(F,F)=2$ and $\chi(E,F)=\chi(F,E)\eqqcolon s$. Another direct computation shows that $t_E t_F = t_F t_E$ if $s=0$, and $t_E t_F t_E = t_F t_E t_F$ if $s=\pm1$. 
\end{remark}

\subsection{The spherical subcategory $\boldsymbol{\DD_F}$} \label{sub:spherical-subcategory}

We define the \emph{spherical subcategory $\DD_F$} and the \emph{asphericity subcategory $\QQ_F$} of $F$ as
\[ \DD_F \coloneqq  \lorth Q_F, \qquad \QQ_F \coloneqq  \DD_F\orth = (\lorth Q_F)\orth , \]
these are full, triangulated subcategories of $\DD$.
For the asphericity subcategory, we have the inclusion
 $\sod{Q_F} \subset \QQ_F$.
If $Q_F$ is an exceptional object, then the two categories coincide. This will occur in examples considered below, but we will also encounter cases where the inclusion is strict.
By Lemma~\ref{lem:Hom(F,Q)=0}, $\DD_F$ contains $F$ and, by Lemma~\ref{lem:subcat-dicho}, the twist functor $\TTT_F\colon \DD_F \to \DD_F$ exists.

\begin{theorem} \label{thm:spherelike-CY}
Let $\DD$ be a $\kk$-linear, Hom-finite triangulated category. 
If $F$ is a $d$-spherelike object of $\DD$ with Serre dual, then $F$ is $d$-Calabi-Yau and hence
$d$-spherical in $\DD_F$.
\end{theorem}

\begin{proof}
We want to show that $\Hom(A,F) \cong \Hom(F,A[d])^*$ for all $A \in \DD_F$.
For $A\in\DD_F=\lorth Q$, we apply $\Hom(A,\blank)$ to the triangle $F \to \omega(F) \to Q$,
and get two isomorphisms $\Hom(A,F) \cong \Hom(A,\omega(F)) \cong \Hom(F,A[d])^*$ each of them functorial in $A$ (the second one by definition of Serre duals).
\end{proof}

\begin{corollary} \label{cor:twist_adjoints}
Let $F\in\DD$ be as above and assume that $\DD$ is $\Hom^\bullet$-finite, idempotent complete and algebraic. Then $\TTT_F$ induces auto\-equivalences of $\QQ_F$ and $\DD_F$.
In fact, $\TTT_F\colon \DD_F \isom \DD_F$ is the restriction of $\TTT_F\colon\DD\to\DD$ to $\DD_F$.

If moreover $\DD$ possesses a Serre functor, then the left adjoint $\TTl_F$  induces autoequivalences of $\DD_F$ and of $\SSS\inv(\QQ_F)$, and the right adjoint $\TTr_F$  induces autoequivalences of $\QQ_F$ and of $\SSS(\DD_F)$.
\end{corollary}

\begin{proof}
By the assumptions on $\DD$, the twist functor $\TTT_F\colon\DD\to\DD$ is well-defined. Moreover,
$\TTT_F$ induces endofunctors of $\DD_F$ and $\QQ_F$ by Lemmata~\ref{lem:subcat-dicho} and \ref{lem:Hom(F,Q)=0}.
The restriction $\TTT_F|_{\DD_F}\colon\DD_F\to\DD_F$ is an autoequivalence by the theorem and Lemmata~\ref{lem:spanningprop}(4) and \ref{lem:subcat-dicho}.
Moreover, $\QQ_F=\DD_F\orth$ and $F\in\DD_F$ implies $\TTT_F|_{\QQ_F} = \id$. So $\TTT_F$ quite trivially is an autoequivalence of $\QQ_F$.

We turn to $\TTl_F$ which sits in the triangle $\TTl_F \to \id \to \Hom^\bullet(\blank,F)^* \otimes F$. Plugging an object $A\in\DD_F$ into this triangle and applying $\Hom^\bullet(\blank,Q_F)$, we find that $\TTl_F$ induces an endofunctor of $\DD_F$.
There it is still the left adjoint of the autoequivalence $\TTT_F|_{\DD_F}$, so $\TTl_F|_{\DD_F}$ is also an autoequivalence.
By the defining triangle, we see that $\TTl_F|_{\SSS\inv(\QQ_F)}$ is just the identity.

The statement about the right adjoint follows from $\TTr_F = \SSS \TTl_F \SSS\inv$.
\end{proof}

Now we prove that $\DD_F$ is indeed the maximal subcategory containing $F$ as a spherical object:

\begin{theorem} \label{thm:maximality}
Let $\UU\subset\DD$ be a full, triangulated subcategory and $F\in\UU$ a $d$-spherical object.
If $F$ has a Serre dual in $\DD$, then $\UU\subset\DD_F$.
\end{theorem}

\begin{proof}
The statement is an extension of Lemma~\ref{lem:Hom(F,Q)=0}: here we want to show $\Hom^\bullet(U,Q)=0$ for all $U\in\UU$. This is equivalent to $\UU\subset\DD_F=\lorth Q_F$. 

The proof will use cohomological functors, i.e.\ contravariant functors from $\kk$-linear triangulated categories to $\kk$-vector spaces, mapping triangles to long exact sequences. For $D\in\DD$, we put 
 $h_D      \coloneqq  \Hom_\DD(\blank,D) \colon \DD\op \to \kk\hh\mod$, and
 $h_D|_\UU \coloneqq  \Hom_\DD(\blank,D) \colon \UU\op \to \kk\hh\mod$
for the induced functor on $\UU$.

Using the Serre duals of $F$ in $\DD$ and $\UU$, there are isomorphisms
\begin{align*}
  h_{\SSS F[-d]} = \Hom_\DD(\blank,\SSS F[-d]) &\cong \Hom_\DD(F[-d],\blank)^* ,\\
  h_F|_\UU      = \Hom_\UU(\blank,F) &\cong \Hom_\UU(F[-d],\blank)^* ,
\end{align*}
where the second line uses $F\in\UU$ and that $\UU$ is a full subcategory.
Hence we obtain an isomorphism of cohomological functors $g\colon h_{\SSS F[-d]}|_\UU\isom h_F|_\UU$.
In general, $h_{\SSS F[-d]}|_\UU$ is not representable as a functor on $\UU$. However, $h_F|_\UU$ is representable due to $F\in\UU$, allowing to invoke the Yoneda lemma:
\[ \Hom_{\text{Fun}}(h_F|_\UU,h_{\SSS F[-d]}|_\UU) = h_{\SSS F[-d]}|_\UU(F) = \Hom_\DD(F,\SSS F[-d]) . \]
The morphism $w\colon F\to\omega(F)=\SSS F[-d]$ induces a natural transformation $w_*\colon h_F\to h_{\SSS F[-d]}$. Our aim is to show that $w_*|_\UU\colon h_F|_\UU\to h_{\SSS F[-d]}|_\UU$ is a functor isomorphism --- assuming this is true, the induced triangles
 $\Hom^\bullet(U,F) \xrightarrow{w_*} \Hom^\bullet(U,\SSS F[-d]) \to \Hom^\bullet(U,Q_F)$
then immediately enforce $\Hom(U,Q_F)=0$ for all $U\in\UU$. 

Now, to show that $w_*|_\UU \colon h_F|_\UU \to h_{\SSS F[-d]}|_\UU$ is a functor isomorphism, we can equivalently check  $g\circ w_*|_\UU\colon h_F|_\UU \to h_F|_\UU$.
By the Yoneda lemma, latter is an isomorphism if and only if the corresponding map in $\Hom_\UU(F,F)$ is, which is given by $\left(g\circ w_*|_\UU\right)(F)(\id_F)$.
Unraveling this and using that $g$ is an isomorphism, we are left to show this for $w_*\colon \Hom(F,F) \to \Hom(F,\omega(F))$, which which was shown in the proof of Lemma~\ref{lem:Hom(F,Q)=0}.
\end{proof}

\subsection{Assumption $\boldsymbol{\DD=\sod{\CC\orth,\CC}}$ with $\boldsymbol{F\in\CC}$ spherical.}

Under some abstract assumptions, quite a bit can be said about the spherical subcategory. 
Here, we consider:

\assumption{\SOT}{
$F$ is $d$-spherical in a right admissible subcategory $ \CC \xxembed{\iota} \DD$ and $\DD$ has a Serre functor.
}

\noindent
We recall the simple fact that the right adjoint to an inclusion is a left inverse, i.e.\
$\CC\xxembed{\iota}\DD\xxrightarrow{\pi}\CC$ is the identity; see \cite[Rem.~1.24]{Huybrechts}.

Whenever \SOT\ holds, the following theorem allows to compute the spherical subcategory without recourse to the asphericity. This will be used in many of the examples. However, there are also interesting examples not of this type.

\begin{theorem} \label{thm:projection_functor}
Let $F\in\CC\subset\DD$ such that \SOT\ holds. Then $F$ is $d$-spherelike as an object of $\DD$, and the spherical subcategory has a weak semi-orthogonal decomposition 
 $\DD_F = \sod{\CC\orth\cap \lorth F, \; \CC}$.
\end{theorem}

\begin{proof}
It is immediate that $\iota F$ is $d$-spherelike in $\DD$. 
We get $\CC\subset\DD_{\iota} F=\lorth Q_{\iota} F$ from Theorem~\ref{thm:maximality}. 

For $A\in\CC\orth$, we have
 $\Hom^\bullet_\DD(A,\SSS_\DD\iota F[-d])=\Hom^\bullet_\DD(\iota F[-d],A)^*=0$, as $F\in\CC$. Applying $\Hom^\bullet_\DD(\iota A,\blank)$ to the asphericity triangle of $\iota F$ then shows
 $\CC\orth\cap \lorth Q_{\iota F}=\CC\orth\cap \lorth \iota F$.

Now, any object $A \in \DD_{\iota F} \subset \DD$ has a decomposition $A_\CC \to A \to A_\perp$ with $A_\CC \in \CC$ and $A_\perp \in \CC\orth$. 
As shown above $A_\perp \in \CC\orth\cap \lorth \iota F$,
so $\DD_{\iota F} = \sod{\CC\orth\cap \lorth \iota F,\; \CC}$
is a weak semi-orthogonal decomposition.
This is the formula of the theorem, where by abuse of notation we have identified $F$ with $\iota F$ as an object of $\DD$.
\end{proof}

\begin{remark}
\label{rem:0sph-dec}
Let $F = E \oplus E'$ be the 0-spherelike object obtained from two mutually orthogonal, exceptional objects $E$ and $E'$. Then $F\in\CC \coloneqq  \sod{E,E'}$ and the inclusion $\CC \xxembed{\iota} \DD$ has right adjoint $\pi \coloneqq  \TTT_E \oplus \TTT_{E'}$.
Since $\pi \iota=\id_\CC$, we can apply the proposition and get $\DD_F = \sod{\orth F \cap F\orth,\sod{F}}$.

Obviously, $\DD_F \supset \sod{\orth F\cap F\orth,\sod{F}}$ holds for all spherelike objects $F$.
However, the inclusion is strict in general. A simple example is given by the 1-spherical skyscraper sheaf $\kk(p)$ for a point $p$ on a smooth curve $C$. Then $\DD_{\kk(p)} = \Db(C)$ as $\kk(p)$ is spherical, but $\sod{\orth \kk(p) \cap \kk(p)\orth,\; \kk(p)}$ only contains objects with zero-dimensional support; see Example~\ref{ex:pulledback-ruled}.
\end{remark}

\subsection{Assumption $\boldsymbol{\DD = \sod{\QQ_F,\DD_F}}$}

We consider this condition:

\assumption{\SOD}{
$\DD$ has a Serre functor and $\DD_F\embed\DD$ is right admissible.
}

\noindent
As a direct consequence, we get a weak semi-orthogonal decomposition $\DD = \sod{\QQ_F,\DD_F}$.
Furthermore, \SOD\ gives semi-orthogonal decompositions $\sod{\DD_F,\SSS\inv(\QQ_F)}$ and $\sod{\SSS(\DD_F),\QQ_F}$ of $\DD$; therefore Lemma~\ref{lem:piecewise-invertible-is-conservative}, Theorem~\ref{thm:spherelike-CY} and Corollary~\ref{cor:twist_adjoints} immediately imply

\begin{proposition}
Assume \SOD\ and that $\DD$ is algebraic, idempotent complete and $\Hom^\bullet$-finite. Let $F\in\DD$ be a $d$-spherelike object. Then the twist functor $\TTT_F$ and its adjoints are piecewise invertible and in particular conservative.
\end{proposition}

\section{Examples from algebraic geometry} \label{sec:AG}

\noindent
We will always work with smooth, projective varieties over an algebraically closed field $\kk$ and the triangulated category under investigation will be the bounded derived category of coherent sheaves. As is well known, classical Serre duality shows that $\SSS A\coloneqq A\otimes\omega_X[\dim(X)]$ is the Serre functor of $\Db(X)$. Note that $d$-Calabi-Yau objects of $\Db(X)$ must have $d=\dim(X)$. This is why among $d$-spherelike objects of $\Db(X)$ those with $d=\dim(X)$ are particularly interesting and we have $\omega(A)=\SSS A[-d]=\omega_X\otimes A$ for such an object $A$, justifying the notation.

\subsection{Spherelike vector bundles} \label{sec:spherelike-vb}

Let $V$ be a $d$-spherelike locally free sheaf on a variety of dimension $n$. Assuming that $\kk$ has characteristic 0, the endomorphism bundle splits
 $ V\dual \otimes V \cong \OO_X \oplus W $
where $W$ is locally free and self dual, i.e.\ $W\dual \cong W$. We get
\[ \kk\oplus\kk[-d] = \Hom^\bullet(V,V) \cong \Hom^\bullet(\OO_X,V\dual\otimes V) \cong
   H^\bullet(\OO_X) \oplus H^\bullet(W) .\]
Since $H^0(\OO_X)=\kk$ in any case, there are two possibilities:
\begin{itemize}
\item Either $H^\bullet(W)=\kk[-d]$ and $H^\bullet(\OO_X)=\kk$, i.e.\ $\OO_X$ is exceptional,
\item or $H^\bullet(W)=0$ and $H^\bullet(\OO_X)=\kk\oplus\kk[-d]$, i.e.\ $\OO_X$ is $d$-spherelike.
\end{itemize}
Let us restrict to dimension 2. We are therefore interested in surfaces with exceptional or spherelike structure sheaf. Below, we compile a list of those, assuming $\chara(\kk)=0$. Recall that the irregularity $q\coloneqq \dim H^1(\OO_X)$ and the geometric genus $p_g\coloneqq \dim H^2(\OO_X)$ are birational invariants, as is the Kodaira dimension $\kappa$. See \cite{Barth-etal} for these notions and the following list. 
For Kynev surfaces, see \cite{Kynev} or \cite{Park-Park-Shin}.


\medskip
\begin{center}
\noindent
\begin{tabular}{@{}lcp{0.47\textwidth}@{}} \toprule
                                      & $\kappa$  & minimal model (or example) \\ \midrule
$\OO_X$ exceptional $(q=p_g=0)$       & $-\infty$ & rational surfaces \\
                                      & 0         & Enriques surfaces \\
                                      & 1         & (e.g.\ Dolgachev surfaces) \\
                                      & 2         & (e.g.\ Barlow, Burniat, Campedelli, Catanese, Godeaux surfaces) \\ \midrule
$\OO_X$ 1-spherelike $(q=1, p_g=0)$   & $-\infty$ & ruled surfaces of genus 1 \\
                                      & 0         & bielliptic surfaces \\ \midrule
$\OO_X$ 2-spherelike $(q=0, p_g=1)$   & 0         & K3 surfaces \\ 
                                      & 1         & (see below) \\
                                      & 2         & (e.g.\ Kynev surface) \\ \bottomrule
\end{tabular}
\end{center}
\medskip

\noindent
Wherever we write `e.g.' only examples are known and a full classification is not available; those surfaces need not be minimal.
For an example with invariants $\kappa=1, q=0, p_g=1$, see \cite[Ex.~3.3]{Peters}. It is constructed as a double logarithmic transform of a minimal elliptic fibration over $\IP^1$. However, it is not clear whether there are projective examples. See also \cite[Ex.~V.13.2]{Barth-etal}.

We treat structure sheaves of ruled surfaces over elliptic curves in Example~\ref{ex:structure-sheaf-ruled-surface}, 
and $2$-spherelike structure sheaves in Proposition~\ref{prop:structure-sheaf-spherelike}

\subsection{Blowing ups} \label{sec:blowup}

Let $X$ be a variety of dimension $d\geq2$ and $\pi\colon \tX \to X$ the blow-up of $X$ in a point $p$. Denote the exceptional divisor by $R$; we know $R\cong\IP^{d-1}$ and $\OO_R(R)\cong\OO_R(-1)$.
Recall that the derived pullback functor $\pi^*\colon \Db(X)\to\Db(\tX)$ is fully faithful and that the canonical bundle of the blow-up is given by $\omega_\tX=\pi^*\omega_X\otimes\OO_\tX(d'R)$ where for notational purposes we set $d'\coloneqq d-1$ in this section.

The derived category of $\Db(\tX)$ has a semi-orthogonal decomposition
\[ \Db(\tX) = \bigsod{ \OO_R(-d'), \ldots, \OO_R(-1), \pi^* \Db(X)} , \]
where we note that $\OO_R(-d'),\ldots,\OO_R(-1)$ is an exceptional sequence; see \cite[\S11.2]{Huybrechts}. Let $S\in\Db(X)$ be a spherical object and $F \coloneqq \pi^*S \in \Db(\tX)$ its $d$-spherelike pull-back. Assumption \SOT\ holds, so that Theorem~\ref{thm:projection_functor} applies.

\begin{lemma} \label{lem:blowup-asphericity}
Let $S\in\Db(X)$ be spherical and $F=\pi^*S\in\Db(\tX)$ its spherelike pull-back. Then $F$ has asphericity $Q_F=F\otimes\OO_{d'R}(d'R)$. Furthermore, $F$ is spherical if and only if $p\notin\supp(S)$.
\end{lemma}

\begin{proof}
First, assume $p\notin\supp(S)$. Then, $F\otimes\omega_\tX\cong F$, since $\supp(F)\cap R=\varnothing$ and $\OO_\tX(d'R)$ is trivial off $R$.

Now we turn to the asphericity $Q$. We can assume $p\in\supp(S)$ --- otherwise, $Q=0$, in compliance with the claimed formula. Now observe
 $F \otimes \omega_\tX = \pi^*S \otimes \pi^*\omega_X \otimes \OO_\tX(d'R) 
                       = F \otimes \OO_\tX(d'R) ,$
using the formula for $\omega_\tX$ and the Calabi-Yau property $S\otimes\omega_X=S$.
Tensoring the exact sequence
 $0\to \OO_\tX \to \OO_\tX(d'R) \to \OO_{d'R}(d'R) \to 0$
with $F$ gives the triangle
\[ F \to F\otimes \OO_\tX(d'R) \to F\otimes\OO_{d'R}(d'R) \]
where we recall that the tensor product of the last term is derived. Note that the first map must be non-zero --- otherwise $F\otimes\OO_{d'R}(d'R)$ would be a direct sum $F[1] \oplus F\otimes\OO_\tX(d'R)$, contradicting that $F\otimes\OO_{d'R}(d'R)$ is supported on $R$ but the support of $F=\pi^*S$ is strictly bigger than $R$: spherical objects are not supported on points if $d\geq2$.
As $\Hom(F,\omega(F))=\Hom(F,F\otimes\OO_\tX(d'R))$ is one-dimensional, the above triangle is therefore isomorphic to the triangle defining the asphericity, $F\to\omega(F)\to Q_F$.

Finally, we show that $F$ spherical implies $p\notin\supp(S)$. Thus we have $0 = Q_F = F\otimes\OO_{d'R}(d'R)$ and then $0=F\otimes\OO_{d'R}=\pi^*S\otimes\OO_{d'R}$. Applying $\pi_*$ and the projection formula, we get $0=S\otimes\pi_*\OO_{d'R}$. Now $\pi_*\OO_{d'R}$ is supported on $p$, and is non-zero (the sheaves $\OO_{iR}$ have global sections for all $i\geq0$ since $\OO_R(-iR)=\OO_R(i)$ do). Therefore, $0=S\otimes\pi_*\OO_{d'R}$ implies $p\notin\supp(S)$.
\end{proof}

\begin{proposition} \label{prop:sod-equality}
Let $\pi\colon\tX\to X$ be the blowing up of a smooth projective variety of dimension $d$ in a point $p$. If $S\in\Db(X)$ is a spherical object with $p\in\supp(S)$, then for $F\coloneqq\pi^*S$ assumption \SOD\ holds true and, moreover
there is a refinement of semi-orthogonal decompositions
\[ \bigsod{\OO_R(-(d-1)),\ldots,\OO_R(-1),\pi^*\Db(X)} \prec \bigsod{\QQ_F,\DD_F} \]
with $\DD_F = \pi^*\Db(X)$ and $\QQ_F = \sod{\OO_R(-(d-1)),\ldots,\OO_R(-1)}$.
\end{proposition}

\begin{proof}
Again we put $d'=d-1$ for the sake of readability. As another temporary notation, put $\EE\coloneqq \sod{\OO_R(-d'),\ldots,\OO_R(-1)}$, so that $\Db(\tX)=\sod{\EE,\pi^*\Db(X)}$.
From Theorem~\ref{thm:projection_functor}, we know 
 $\DD_F = \sod{\EE\cap \lorth F,\pi^*\Db(X)}$.
Our goal is to prove $\EE\cap \lorth F=0$. Since $\EE$ is generated by an exceptional sequence, we are reduced to showing $\OO_R(-i)\notin\lorth F$ for $i=1,\ldots,d'$. 

Fix such an $i$ and proceed
\begin{align*}
   \Hom^\bullet(\OO_R(-i),F) 
&= \Hom^\bullet(F,\OO_R(-i)\otimes\omega_\tX[d])^* \\
&= \Hom^\bullet(\pi^*S,\OO_R(-i)\otimes\pi^*\omega_X\otimes\OO_\tX(d'R))^*[-d] \\
&= \Hom^\bullet(\pi^*(S\otimes\omega_X\inv),\OO_R(-i)\otimes\OO_\tX(d'R))^*[-d] \\
&= \Hom^\bullet(\pi^*S,\OO_R(-i-d'))^*[-d] \\
&= \Hom^\bullet(S,\pi_*\OO_R(-i-d'))^*[-d] \\
&= \Hom^\bullet(S,\kk(p)\otimes H^\bullet(\OO_R(-i-d')))^*[-d] ,
\end{align*}
using Serre duality, the formula for $\omega_\tX$, the Calabi-Yau property of $S$, the relation $\OO_R(R)=\OO_R(-1)$ and adjunction $\pi^*\dashv\pi_*$. About the equality used in the closing step, $\pi_*\OO_R(-i-d') = \kk(p) \otimes H^\bullet(\OO_R(-i-d'))$: the two maps $R\embed\tX\xxrightarrow{\pi} X$ and $R\to\{p\}\embed X$ coincide, and so give a commutativity relation of direct image functors. The cohomology is non-zero due to $H^{d'}(\OO_{\IP^{d'}}(-i-d'))\neq0$, for $i>0$.

With $S$ supported on $p$, i.e.\ $\Hom^\bullet(S,\kk(p))\neq0$, we finally assemble these pieces into the desired non-orthogonality $\Hom^\bullet(\OO_R(-i),F)\neq0$.
\end{proof}

The proposition can be extended inductively. 

\begin{corollary}
Let $\tX = X_l \xrightarrow{\pi_l} \cdots \xrightarrow{\pi_1} X_0 = X$ be a sequence of blowups in points of a smooth projective variety of dimension $d$. If $S \in \Db(X)$ is spherical and the blowups happen in the support of (the pullback of) $S$, then $\Db_F = \pi^* \Db(X)$ where $\pi$ is the concatenation of the blowups and $F = \pi^* S$. Moreover, \SOD\  holds in this situation.
\end{corollary}

\begin{proof}
Let $\pi_i$ be the blow-up of the point $p_i$ with exceptional divisor $E_i$, and 
write $\pi_{l,k} \coloneqq \pi_l \circ \cdots \pi_k$ for $1\leq k\leq l$.
There is the semi-orthogonal decomposition 
\[
\Db(\tX) = \sod{ \OO_{E_l}(-d'),\ldots,\OO_{E_l}(-1), \pi^*_l(\mathcal{E}), \pi^* \Db(X)}
\]
for a category $\mathcal{E}$ generated by certain $\OO_{E_i}(-k)$ with $1\leq i<l$.
Applying Theorem~\ref{thm:projection_functor}, we have to check that
$\sod{ \OO_{E_l}(-d'), \ldots, \OO_{E_l}(-1), \pi^*_l(\mathcal{E})} \cap \lorth F$ is zero. Note that $\Hom^\bullet(\pi^*_l(\mathcal{E}),F)$ does not vanish by induction. So it remains to show that
$\Hom^\bullet(\OO_{E_l}(-k),F)$ is non-zero for $1\leq k\leq d'$.
Since $\omega_\tX = \pi^* \omega_X \otimes \bigotimes_{i=1}^l \pi^*_{l, i+1} \OO(d'{E_{i}}) $
where $\pi_{l,l+1} \coloneqq  \id$, we compute
\begin{align*}
\Hom^\bullet(\OO_{E_l}(-k),F) &= \Hom^\bullet(S,\pi_*(\OO_{E_l}(-k-d')))^*[-d]\\
&= \Hom^\bullet(S,\kk(p) \otimes H^\bullet(\OO_{E_l}(-k-d')))^*[-d] \neq 0,
\end{align*}
using analogous arguments as in the proof above. Additionally, we have used that $E_l.\pi^*_{l, i+1}(E_i) = 0$ for $i<l$ and moreover, $(\pi_{l-1, 1})_* \kk(p_l) = \kk(p)$ for some point $p$ in the support of $S$. Therefore, we are done.
\end{proof}

\begin{remark}
The assumption that the centers of the blow-ups have to be within the support of (the pullback of) $S$ is not a strong restriction.
In fact, blow-ups outside of $S$ can be performed independently. So without loss of generality, we can perform such blow-ups at first, under which the pullback of $S$ stays spherical by Lemma~\ref{lem:blowup-asphericity}.

Consequently, if a spherelike object is a pullback of a spherical one, then we can recover the derived category of the variety where it is spherical.
\end{remark}

We single out a special case of the corollary which already appeared in the introduction. By definition, $X$ is a Calabi-Yau variety if its structure sheaf $\OO_X$ is spherical. 

\begin{proposition} \label{prop:structure-sheaf-spherelike}
Let $X$ be a Calabi-Yau variety and $\tX\to X$ be any succession of blow-ups in points. Then $\OO_\tX$ is a spherelike sheaf, and its spherical subcategory is $\Db(\tX)_{\OO_\tX}=\Db(X)$.
\end{proposition}

\begin{example} \label{ex:stdex}
Let $X$ be a surface containing a $-2$-curve $C$, i.e.\ a smooth, rational curve $C$ with $C^2=-2$.
Then $S = \OO_C$ is a spherical object in $\Db(X)$; see \cite[Ex.~8.10(iii)]{Huybrechts}. Let $\pi\colon \tX\to X$ be the blow-up of $X$ in a point on $C$. Then $\pi^*S=\OO_{\pi\inv(C)}$ is a 2-spherelike object.

The total transform $\pi\inv(C)=\tilde{C}+R$ is a reducible curve, having as components the strict transform $\tilde{C}$ of $C$ and the exceptional divisor $R$. We remark that $\tilde{C}+R$ has self-intersection $-2$, as follows from $\tilde{C}^2=-3$ and $R^2=-1$. Let us abusively write $C$ instead of $\tilde{C}$ for the strict transform.

We explicitly compute the asphericity $Q$ of the properly 2-spherelike object $F=\OO_{C+R}$. By Lemma~\ref{lem:blowup-asphericity}, it is given by the (derived) tensor product $Q=\OO_{C+R}\otimes\OO_{R}(R)$. Resolving $\OO_{C+R}$ by $i\colon\OO_\tX(-C-R)\to\OO_\tX$,
\begin{align*} 
 Q &= \OO_{C+R} \otimes \OO_R(R) 
    = \OO_R(R) \oplus \OO_\tX(-C-R)\otimes\OO_R(R)[1] \\
   &= \OO_R(-1) \oplus \OO_R(-C)|_R[1] = \OO_R(-1) \oplus \OO_R(-1)[1],
\end{align*}
where we used $i|_R=0$, giving the direct sum, and  $C.R=1$. We conclude $\sod{Q}=\QQ_F$ --- note that that would be wrong without split closure on $\sod{Q}$.
\end{example}

\subsection{Ruled surfaces}
For a different kind of example consider a ruled surface $\pi\colon X\to C$ where $C$ is a smooth, projective curve of arbitrary genus. There is a section which we denote by $C_0\subset X$. It is a classical fact about ruled surfaces that the direct image $V\coloneqq \pi_*\OO_X(C_0)$ is a vector bundle of rank 2 on $C$ (in particular, all higher direct images vanish) with the property $X=\IP(V)$; see \cite[\S{}V.2]{Hartshorne} or \cite[\S5]{Friedman}.

Since ruled surfaces are special cases of projective bundles, we again get a semi-orthogonal decomposition
 $\Db(X)=\sod{\pi^*\Db(C)\otimes\OO_X(-C_0),\; \pi^*\Db(C)}$, 
see \cite[Cor.~8.36]{Huybrechts}.
Here, $\OO_X(C_0)$ is the relatively ample line bundle $\OO_\pi(1)$. This is another situation in which Theorem~\ref{thm:projection_functor} applies.

Given a spherical object $S\in\Db(C)$, its pullback $F\coloneqq \pi^*S$ is 1-spherelike in $\Db(X)$. We know
 $\DD_F = \sod{(\pi^*\Db(C)\otimes\OO_X(-C_0)) \cap \lorth F,\; \pi^*\Db(C)}$
from the theorem. In order to determine the left-hand intersection, take an object $B\coloneqq \pi^*A\otimes\OO_X(-C_0)$ with $A\in\Db(C)$ and carry on with
\begin{align*}
 \Hom^\bullet_X(B,F) &= \Hom^\bullet_X(\pi^*A\otimes\OO_X(-C_0),\pi^*S) \\
                    &= \Hom^\bullet_C(A,S\otimes\pi_*(\OO_X(C_0))) 
                     = \Hom^\bullet_C(A,S\otimes V) .
\end{align*}
We conclude
 $\DD_F = \bigsod{ \pi^*(\lorth(S\otimes V))\otimes\OO_X(-C_0),\; \pi^*\Db(C) }$.

It is well known that $\Db(C)$ has no non-trivial semi-orthogonal decompositions unless $C=\IP^1$, and then assumption \SOD\ cannot be met.

\begin{example} \label{ex:pulledback-ruled}
The skyscraper sheaf $S\coloneqq \kk(p)$ is spherical in $\Db(C)$ for any point $p\in C$. Then $F=\OO_P$ where $P\coloneqq \pi\inv(p)\cong\IP^1$ is the structure sheaf of the fibre over $p$. Here, $S\otimes V=\kk(p)^2$ regardless of the actual surface. We claim that $\lorth\kk(p)=\Db_U(C)$, the  subcategory of objects of $\Db(C)$ supported on the open set $U\coloneqq C\backslash\{p\}$. (This claim follows from standard facts: as $C$ is a smooth curve, every object of $\Db(C)$ is isomorphic to its cohomology complex; every sheaf is a direct sum of its torsion sheaf and the torsion-free quotient, the latter always mapping to any skyscraper sheaf.) Altogether
\[ \DD_{\OO_P} = \bigsod{\pi^*\Db_U(C)\otimes\OO_X(-C_0),\; \pi^*\Db(C)} . \]
We point out that $\pi^*\Db_U(C)\otimes\OO_X(-C_0)$ is generated by $\OO_{\pi\inv(c)}(-1)$ for all $c\in U$, i.e.\ the $(-1)$-twisted structure sheaves of all fibres except $\pi\inv(p)$.

As to the asphericity: $\omega_X\cong\OO_X(-2C_0)\otimes\pi^* L$ for some line bundle $L\in\Pic(C)$. Hence
 $F\otimes\omega_X\cong\pi^*(\kk(p)\otimes L)\otimes\OO_X(-2C_0) = \OO_X(-2C_0)|_P=\OO_P(-2)$.
The triangle defining $Q$ therefore is
 $\OO_P \to \OO_P(-2)[1] \to Q$
so that $Q\cong\OO_P(-1)^2[1]$ --- this is just the Euler sequence for $P\cong\IP^1$.

Hence, assumption \SOD\ is not fulfilled: $\QQ_F=\sod{\pi^*\kk(p)\otimes\OO_X(-C_0)}$ and $\DD_F$ do not generate $\Db(X)$ because $\sod{\kk(p)}$ and $\Db_U(C)$ do not generate $\Db(C)$.

This example shows that $\Db(X)$ can contain infinitely many pairwise incomparable spherical subcategories. See \cite[\S2]{HKP:RT} for a further study of this and related questions.
\end{example}

\begin{example} \label{ex:structure-sheaf-ruled-surface}
Now consider the special case of a ruled surface of genus 1, i.e.\ $C$ is an elliptic curve. Then the structure sheaf $\OO_C$ is 1-spherical in $\Db(C)$, hence its pull-back $\pi^*\OO_C=\OO_X$ is 1-spherelike in $\Db(X)$. By the above general computation, the spherical subcategory is
\[ \DD_{\OO_X} = \bigsod{\pi^*(\lorth V)\otimes\OO_X(-C_0),\; \pi^*\Db(C)} . \]
However, the orthogonal category $\lorth V\subset\Db(C)$ depends on the geometry, i.e.\ the choice of $V$. It is well known that for ruled surfaces over elliptic curves, only three possibilities for $V$ can occur, up to line bundle twists which don't affect $\IP(V)$; see \cite[\S5]{Friedman} or \cite[Thm.~V.2.15]{Hartshorne}:
\begin{itemize}
\item $V=\OO_C\oplus L$ with $L\in\Pic(C)$ of non-negative degree;
\item $V$ is a non-trivial extension of $\OO_C$ by $\OO_C$;
\item $V$ is a non-trivial extension of $\OO_C(p)$ by $\OO_C$ for a point $p\in C$.
\end{itemize}
For example, if $V=\OO_C\oplus L$ with $L\in\Pic^0(C)$, then $\lorth V$ contains all line bundles of degree 0 different from $\OO_C$ and $L$. In particular, the complement is smaller for $V=\OO_C\oplus L$ with $L\neq\OO_C$ than for $V=\OO_C\oplus\OO_C$.
\end{example}

\appendix
\section{0-spherelike objects} \label{app:0-spherelike}

\noindent
Let $\DD$ be a category as in Section~\ref{sec:twist-functors} and let $F\in\DD$ be a $d$-spherelike object. Ignoring the grading, the endomorphism algebra $\Hom^\bullet(F,F) \cong \kk^2$ as a $\kk$-vector space.
As an ungraded $\kk$-algebra, only two cases can occur, since $\kk$ is algebraically closed. We call $F$
\begin{itemize}
\item \emph{nilpotent} if $\Hom^\bullet(F,F)=\kk[\varepsilon]/\varepsilon^2$ where $\varepsilon\colon F \to F[d]$ is unique up to scalars;
\item \emph{disconnected} if $\Hom^\bullet(F,F)=\kk\times\kk$, so that $\id_F = p_1 + p_2$ for two orthogonal idempotents which are unique up to order.
\end{itemize}
To see that these are all cases, let $\{1,b\}$ be a basis of $\Hom^\bullet(F,F)$ and consider the surjection $\kk[x] \onto \Hom^\bullet(F,F)$, $x\mapsto b$. Its kernel is generated by a polynomial of degree $2$ having either just one or two distinct roots, which separates both cases.
We note that the second case can only occur for $d=0$. We want to mention that given two exceptional, mutually orthogonal objects $E$ and $E'$, their direct sum $E\oplus E'$ is a disconnected $0$-spherelike object. If $\DD$ is idempotent complete (for example, if it is a derived category \cite[Cor.~2.10]{Balmer-Schlichting}), then every disconnected spherelike object is of the form above.

\apppart{Proof of Lemma~\ref{lem:spanningprop}, (5)}
If $F$ is nilpotent then $F$ is indecomposable, so the original argument applies. On the other hand, if $F$ is disconnected, i.e.\ $\Hom^\bullet(F,F) \cong \kk\times\kk$ and consequently $d=0$, then $F\cong E_1 \oplus E_2$ for two non-zero objects $E_1$ and $E_2$. With $F$ 0-spherelike, both objects $E_1$ and $E_2$ have to be exceptional and mutually orthogonal. In particular, they are indecomposable, so $F \oplus F \cong \SSS F \oplus \SSS F$ implies $F \cong \SSS F$, using that $\DD$ is Krull-Schmidt. Hence, the claim is established.

\apppart{Construction of $w\colon F\to \omega(F)=\SSS F$ and proof of Lemma~\ref{lem:Hom(F,Q)=0}}
First, we will treat the case of $F$ being nilpotent, so $\End(F)$ has the basis $(\id,\varepsilon)$ as a $\kk$-vector space with $\varepsilon^2=0$.
Let $(\id\dual,\varepsilon\dual)$ be the dual basis of $\End(F)^*$. 
There is a natural structure of $\End(F)^*$ as a right $\End(F)$-module, given by $\varepsilon\dual \cdot \varepsilon = \id\dual$ and $\id\dual \cdot \varepsilon = 0$.
Since $\sigma\colon \Hom(F,\SSS F) \xrightarrow{\sim} \End(F)^*$ is functorial, $\sigma$ is an isomorphism of $\End(F)$-modules.
Hence, there is a basis $(\iota,\phi)$ of $\Hom(F,\SSS F)$ with $\iota \circ\varepsilon = 0$ and $\phi \circ\varepsilon = \iota$.
We can choose $w = a \iota + b \phi$ with $a,b$ arbitrary, as long as $b \neq 0$. 

Using the basis of $\End(F)$ we see that $w_* \colon \End(F) \to \Hom(F,\SSS F)$ is an isomorphism, which proves Lemma~\ref{lem:Hom(F,Q)=0} in this case.

\smallskip
Next, we turn to the case that $F = E_1 \oplus E_2$ is disconnected, so $E_i$ are mutually orthogonal exceptional objects.
By Serre duality we get that
\[
\Hom^\bullet(E_i,\SSS E_j) = 
\begin{cases}
\kk \cdot s_i & \text{if } i=j\\
0 & \text{if } i\neq j
\end{cases}
\]
Therefore any map $F \to \SSS F$ is of the form $a_1 s_1 + a_2 s_2$.
Now choose $w = a_1 s_1 + a_2 s_2$ with $a_i \neq0$.

Again, $w_*\colon \End(F) \to \Hom(F,\SSS F)$ turns out to be an isomorphism, so also in the disconnected case,  Lemma~\ref{lem:Hom(F,Q)=0} is shown.
We note that if one would choose $a_i=0$, then $E_i$ is a direct summand of $Q_F$.

\smallskip
Finally, we want to remark, that in both cases, different choices of $w$ yield isomorphic $Q_F$.
We show in the nilpotent case, that the asphericity $Q_F$ of $w= a\iota + b\phi$ is isomorphic to $Q'_F$ of $w'=\phi$; the disconnected case is similar.
For the corresponding aspherical triangles, we get the following map of triangles, where the right map is the isomorphism between the asphericities:
\[
\xymatrix@R=3ex{
F \ar[r]^{w'} \ar[d]_{\frac1{b}\id-\frac{a}{b^2}\varepsilon} & \SSS F \ar[r] \ar[d]^{\id} & Q'_F \ar@{..>}[d]\\
F \ar[r]_{w}       & \SSS F \ar[r]             & Q_F
}
\]

\addtocontents{toc}{\protect\setcounter{tocdepth}{-1}}
\section*{Acknowledgements}

\noindent
Martin Kalck is grateful for the support by the DFG grant Bu-1886/2-1. The authors would like to thank Stephen Coughlan, Sergey Galkin, Matthias Sch\"utt, Pawel Sosna, Greg Stevenson, Rahbar Virk and Michael Wemyss for answering our questions. We are very grateful to the anonymous referees for very thorough reading and many valuable comments.

\bigskip
\noindent
\resizebox{\textwidth}{!}{\emph{Contact:} \texttt{ahochene@math.uni-koeln.de, m.kalck@ed.ac.uk, ploog@math.uni-hannover.de}}

\end{document}